\documentclass[a4,12pt]{article}%

\usepackage{graphicx}
\usepackage{graphics}
\newcounter{figurecounter}
\usepackage{amsmath}
\usepackage{amsfonts}
\usepackage{amssymb}
\usepackage{enumerate}
\newtheorem{theorem}{Theorem}

\newtheorem{definition}{Definition}

\newtheorem{lemma}{Lemma}

\newtheorem{notation}{Notation}

\newtheorem{proposition}{Proposition}

\newenvironment{proof}[1][Proof]{\textbf{#1.} }{\ \rule{0.5em}{0.5em}}
\newcommand{\E}{{\rm \bf E}}

\newcommand{\prob}{{\rm \bf P}}

\newcommand{\dR}{{\bf R}}

\newcommand{\calB}{{\cal B}}

\newcommand{\calF}{{\cal F}}

\newcommand{\calV}{{\cal V}}

\begin{document}

\title{The Value Functions of Markov Decision Processes%
\thanks{Lehrer acknowledges the support of the Israel Science Foundation, Grant \#963/15.
Solan acknowledges the support of the Israel Science Foundation, Grants \#323/13.}}

\author{Ehud Lehrer\thanks{School of Mathematical Sciences, Tel Aviv University, Tel Aviv 6997800, Israel and INSEAD, Bd.~de Constance, 77305 Fontainebleau
Cedex, France. e-mail: \textsf{lehrer@post.tau.ac.il}.},
Eilon Solan\thanks{School of Mathematical Sciences, Tel Aviv University, Tel
Aviv 6997800, Israel. E-mail: \textsf{eilons@post.tau.ac.il}.},
and Omri N. Solan\thanks{School of Mathematical Sciences, Tel Aviv University, Tel
Aviv 6997800, Israel. E-mail: \textsf{omrisola@post.tau.ac.il}.}}

\date{\today}

\maketitle

\begin{abstract}
We provide a full characterization of the set of value functions of Markov decision processes.
\end{abstract}

\section{Introduction}

Markov decision processes are a standard tool for studying dynamic optimization problems.
The discounted value of such a problem is the maximal total discounted amount that the decision maker can guarantee to himself.
By Blackwell (1965), the function $\lambda \mapsto v_\lambda(s)$  that assigns the discounted value at the initial state $s$ to each discount factor $\lambda$
is the maximum of finitely many rational functions (with real coefficients).
Standard arguments show that
the roots of the polynomial in the denominator of these rational functions
either lie outside the unit ball in the complex plane, or on the boundary of the unit ball, in which case they have multiplicity 1.
Using the theory of eigenvalues of stochastic matrices one can show that
the roots on the boundary of the unit ball must be unit roots.

In this note we prove the converse result:
every function $\lambda \mapsto v_\lambda$ that is the maximum of finitely many rational functions
such that each root of the polynomials in the denominators either lies outside the unit ball in the complex plane, or is a unit root with multiplicity 1
is the value of some Markov decision process.

\section{The Model and the Main Theorem}

\begin{definition}
A \emph{Markov decision process} is a tuple $(S,A,r,q)$ where
\begin{itemize}
\item   $S$ is a finite set of states.
\item   A distribution $\mu \in \Delta(S)$ according to which the initial state is chosen.\footnotemark
\item   $A = (A(s))_{s \in S}$ is the family of sets of actions available at each state $s \in S$.
Denote $SA := \{(s,a) \colon s \in S, a \in A(s)\}$.
\item   $r : SA \to \dR$ is a payoff function.
\item   $q : SA \to \Delta(S)$ is a transition function.
\end{itemize}
\end{definition}
\footnotetext{For every finite set $X$, the set of probability distributions over $X$ is denoted $\Delta(X)$.}

The process starts at an initial state $s_1 \in S$, chosen according to $\mu$. It then evolves in discrete time:
at every stage $n \in \mathbb{N}$ the process is in a state $s_n \in S$,
the decision maker chooses an action $a_n \in A(s_n)$, and a new state $s_{n+1}$ is chosen according to $q(\cdot \mid s_n,a_n)$.

A \emph{finite history} is a sequence $h_n = (s_1,a_1,s_2,a_2,\cdots, s_n) \in H := \cup_{k = 0}^\infty (SA)^k \times S$.
A \emph{pure strategy} is a function $\sigma : H \to \cup_{s \in S} A(s)$ such that $\sigma(h_n) \in A(s_n)$ for every
finite history $h_n = (s_1,a_1,\cdots,s_n)$,
and a \emph{behavior strategy} is a function $\sigma : H \to \cup_{s \in S} \Delta(A(s))$ such that $\sigma(h_n) \in \Delta(A(s_n))$ for every such finite history.
The set of behavior strategies is denoted $\calB$.
In other words, $\sigma$ assign to every finite history a distribution over $A$, which we call a \emph{mixed action}.
A strategy is \emph{stationary} if for every finite history $h_n = (s_1,a_1,\cdots,s_n)$,
the mixed action $\sigma(h_n)$ is a function of $s_n$ and is independent of $(s_1,a_1,\cdots,a_{n-1})$.
Every behavior strategy together with a prior distribution $\mu$ over the state space induce a probability distribution $\prob_{\mu,\sigma}$
over the space of infinite histories $(SA)^\infty$ (which is endowed with the product $\sigma$-algebra).
Expectation w.r.t.~this probability distribution is denoted $\E_{\mu,\sigma}$.

For every discount factor $\lambda \in [0,1)$, the \emph{$\lambda$-discounted payoff} is
\[ \gamma_\lambda(\mu,\sigma) := \E_{\mu,\sigma}\left[\sum_{n=1}^\infty \lambda^{n-1}r(s_n,a_n)\right]. \]
When $\mu$ is a probability measure that is concentrated on a single state $s$ we denote the $\lambda$-discounted payoff also by $\gamma(s,\sigma)$.
The \emph{$\lambda$-discounted value} of the Markov decision process, with the prior $\mu$ over the initial state is
\begin{equation}
\label{equ:value}
v_\lambda(\mu) := \sup_{\sigma \in \calB} \gamma_\lambda(\mu,\sigma).
\end{equation}
A behavior strategy is \emph{$\lambda$-discounted optimal} if it attains the maximum in (\ref{equ:value}).

Denote by $\calV$ the set of all functions $\lambda \mapsto v_\lambda(\mu)$
that are the value function of some Markov decision processes starting with some prior $\mu \in \Delta(S)$.
The goal of the present note is to characterize the set $\calV$.
\begin{notation}

\noindent (i) Denote by $\calF$ the set of all rational functions $\tfrac{P}{Q}$ such that each root of $Q$ is
(a) outside the unit ball, or
(b) a unit root%
\footnote{Recall that a complex number $\omega \in \mathbb{C}$ is a \emph{unit root} if there exists $n \in \mathbb{N}$ such that $\omega^n=1$.}
with multiplicity 1.\\
  (ii) Denote by $Max\calF$ the set of functions that are the maximum of a finite number of functions in $\calF$.
\end{notation}
The next proposition states that any function in $\calV$ is the maximum of a finite number of functions in $\calF$

\begin{proposition}
\label{prop:1}
$\calV\subseteq Max\calF$ .
\end{proposition}

\begin{proof}
By Blackwell (1965), for every $\lambda \in [0,1)$ there is a $\lambda$-discounted pure stationary optimal strategy.
Since the number of pure stationary strategies is finite,
it is sufficient to show that the function $\lambda \mapsto \gamma_\lambda(\mu,\sigma)$
is in $\calF$, for every pure stationary strategy $\sigma$.
For every pure stationary strategy $\sigma$, every prior $\mu$, and every discount factor $\lambda \in [0,1)$,
the vector $(\gamma_\lambda(s_1,\sigma))_{s_1\in S}$ is the unique solution of a system of $|S|$ linear equations%
\footnote{The result is valid for every stationary strategy, not necessarily pure.
We will need it only for pure stationary strategies.}
in $\lambda$:
\[ \gamma_\lambda(s,\sigma) = r(s,\sigma(s)) + \lambda\sum_{s' \in S} q(s' \mid s,\sigma(s))\gamma_\lambda(s',\sigma), \ \ \ \forall s \in S, \]
where $r(s,\sigma(s)) := \sum_{a \in A(s)} \sigma(a \mid s) r(s,a)$
and $q(s' \mid s,\sigma(s)) := \sum_{a \in A(s)} \sigma(a \mid s) q(s' \mid s,a)$
are the multilinear extensions of $r$ and $q$, respectively.
It follows that
\[ \gamma_\lambda(\cdot,\sigma) = (I - \lambda Q(\cdot,\sigma(\cdot)))^{-1} \cdot r(\cdot,\sigma(\cdot)), \]
where $Q(\cdot,\sigma(\cdot)) = Q = (Q_{s,s'})_{s,s' \in S}$ is the transition matrix induced by $\sigma$, that is, $Q_{s,s'} = q(s' \mid s,\sigma(s))$.
By Cramer's rule, the function $\lambda \mapsto (I - \lambda Q)^{-1}$ is a rational function
whose denominator is $\det(I-\lambda Q)$.
In particular, the roots of the denominator are the inverse of the eigenvalues of $Q$.
Since the denominator is independent of $s$, it is also the denominator of $\gamma_\lambda(\mu,\sigma) = \sum_{s \in S} \mu(s)\gamma_\lambda(s,\sigma)$.

Denote the expected payoff at stage $n$ by $x_n := \E_{\mu,\sigma}[r(s_n,\sigma(h_n))]$,
so that $\gamma_\lambda(\mu,\sigma) = \sum_{n=1}^\infty x_n \lambda^{n-1}$.
Since $|x_n| \leq \max_{s \in S, a \in A(s)} |r(s,a)|$ for every $n \in \mathbb{N}$,
it follows that the denominator $\det(I-\lambda Q(\cdot, \sigma(\cdot))$ does not have roots in the interior of the unit ball
and that all its roots that lie on the boundary of the unit ball have multiplicity 1.
Moreover, by, e.g., Higham and Lin (2011) it follows that the roots that lie on the boundary of the unit ball must be unit roots.
\end{proof}

\bigskip

The main result of this note is that the converse holds as well.
\begin{theorem}
\label{theorem:1}
A function $w : [0,1) \to \mathbb{R}$ is in $\calV$ if and only if it is the maximum of a finite number of functions in $\calF$.
\end{theorem}

To avoid cumbersome notations we write $f(\lambda)$ for the function $\lambda \mapsto f(\lambda)$.
In particular, $\lambda f(\lambda)$ will denote the function $\lambda \mapsto \lambda f(\lambda)$.

We start with the following observation.
\begin{lemma}\label{lemma:1}
If $f,g \in \calV$ then $\max\{f,g\} \in \calV$.
\end{lemma}
\begin{proof}
Let $M_f = (S_f,\mu_f,A_f,r_f,q_f)$ and $M_g=(S_g,\mu_g,A_g,r_g,q_g)$ be the Markov decision processes that implement $f$ and $g$ respectively.
To implement $\max\{f,g\}$, define a Markov decision process that contains $M_f$, $M_g$, and an additional state $s^*$,
in which the decision maker chooses one of $M_f$ and $M_g$.
Formally,
let $M = (S_f\cup S_g\cup\{s^*\},A',r',q')$ be a Markov decision process
in which $A'$, $r'$, and $q'$ coincide with $A_f$, $r_f$, and $q_f$ (resp.~$A_g$, $r_g$, and $q_g$) on $S_f$ (resp.~$S_g$),
and in which in the state $s^*$ the decision maker chooses whether to follow $M_f$ or $M_g$,
and the payoff and transitions at that state are the expectation of the payoff and transitions
at the first stage in $M_f$ (resp.~$M_g$) according to $\mu_f$ (resp.~$\mu_g$):
\[ A'(s^*) := \left(\times_{s \in S} A_f(s)\right) \times \left(\times_{s \in S} A_g(s)\right) \times \{f,g\},
\]
and for every $a_f \in \times_{s \in S} A_f(s)$ and every $a_g \in \times_{s \in S} A_g(s)$,
\begin{eqnarray*}
r'(s^*,(a_f,a_g,x)) &:=& \left\{
\begin{array}{lll}
\sum_{s \in S} \mu_f(s) r_f(s,a_f(s)) & \ \ \ \ \ & x=f,\\
\sum_{s \in S} \mu_g(s) r_g(s,a_f(s)) & \ \ \ \ \ & x=g.
\end{array}
\right.\\
q'(\cdot\mid s^*,(a_f,a_g,x)) &:=& \left\{
\begin{array}{lll}
\sum_{s \in S} \mu_f(s) q_f(\cdot\mid (s_{1,f},a_f(s))) & \ \ \ \ \ & x=f,\\
\sum_{s \in S} \mu_g(s) q_g(\cdot\mid (s_{1,g},a_g(s))) & \ \ \ \ \ & x=g.
\end{array}
\right.
\end{eqnarray*}
The reader can verify that the value function of $M$ at the initial state $s^*$ is $\max\{f,g\}$.
\end{proof}

\section{Degenerate Markov Decision Processes}

As mentioned earlier, for every pure stationary strategy $\sigma$ and every initial state $s_1$,
the function $\lambda \mapsto \gamma_\lambda(s_1,\sigma)$ is a rational function of $\lambda$.
Since there is a $\lambda$-discounted pure stationary optimal strategy,
and since there are finitely many such functions,
it follows that the function $\lambda \mapsto v_\lambda$ is the maximum of finitely many rational functions,
each of which is the payoff function of some pure stationary strategy.

When the decision maker follows a pure stationary strategy,
we are reduced to a Markov decision process in which there is a single action in each state.
This observation leads us to the following definition.

A Markov decision process is \emph{degenerate} if $|A(s)|=1$ for every $s \in S$.
When $M$ is a degenerate Markov decision process we omit the reference to the action in the functions $r$ and $q$.
A degenerate Markov decision process is thus  a quadruple $(S, \mu, r,q)$, where $S$ is the state space,
 $\mu$ is a probability distribution over $S$, $r:S\to \mathbb{R}$, and $q(\cdot \mid s)$ is a probability distribution for every state $s\in S$.

Denote by $\calV_D$ the set of all functions
that are payoff functions of some degenerate Markov decision process.
Plainly, $\calV_D \subseteq \calV$.
To prove Theorem \ref{theorem:1}, we first prove that degenerate Markov decision processes implement all functions in $ \calF$.
\begin{theorem}
\label{theorem:2}
$ \calF \subseteq \calV_D $.
\end{theorem}

Theorem \ref{theorem:2} and Lemma \ref{lemma:1} show that $ Max\calF \subseteq \calV$, and together with Proposition \ref{prop:1}
they establish Theorem \ref{theorem:1}.
The rest of the paper is dedicated to the proof of Theorem \ref{theorem:2}.

\subsection{Characterizing the set $\calV_D
$}

The following lemma lists several properties of the functions implementable by degenerate Markov decision processes.
\begin{lemma}
\label{lemma:basic}
For every $f \in \calV_D$ we have
\begin{enumerate}
\item[a)]   $af(\lambda) \in \calV_D$ for every $a \in \mathbb{R}$.
\item[b)]   $f(-\lambda) \in \calV_D$.
\item[c)]   $\lambda f(\lambda) \in \calV_D$.
\item[d)]   $f(c\lambda) \in \calV_D$ for every $c \in [0,1]$.
\item[e)]   $f(\lambda) + g(\lambda) \in \calV_D$ for every $g \in \calV_D$.
\item[f)]   $f(\lambda^n) \in \calV_D$ for every $n \in \mathbb{N}$.
\end{enumerate}
\end{lemma}

\begin{proof}
Let $M_f = (S_f,\mu_f,r_f,q_f)$ be a degenerate Markov decision process whose value function  is $f$.

To prove Part (a), we multiply all payoffs in $M_f$ by $a$.
Formally, define a  degenerate Markov decision process $M' = (S_f,\mu_f,r',q_f)$ that differs from $M$ only in its payoff function:
$r'(s) := ar_f(s)$ for every $s \in S_f$.
The reader can verify that the value function of $M'$ at the initial state $s_{1,f}$ is $af(\lambda)$.

To prove Part (b), multiply the payoff in even stages by $-1$.
Formally, let $\widehat S$ be a copy of $S_f$; for every state $s \in S_f$ we denote by $\widehat s$ its copy in $\widehat S$.
Define a  degenerate Markov decision process $M' = (S_f \cup \widehat S,\mu_f,r',q')$ with initial distribution $\mu_f$ (whose support is $S_f$)
that
visits states in $\widehat S$ in even stages and states in $S_f$ in odd stages as follows:
\begin{eqnarray*}
&&r'(s) := r_f(s), \ \ \ r'(\widehat s) := -r_f(s), \ \ \ \forall s \in S_f,\\
&&q'(\widehat s' \mid s) = q'(s' \mid \widehat s) := q_f(s' \mid s), \ \ \ \forall s,s' \in S_f,\\
&&q'(s' \mid s) = q'(\widehat s' \mid \widehat s) := 0, \ \ \ \forall s,s' \in S_f.
\end{eqnarray*}
The reader can verify that the value function of $M'$ is $f(-\lambda)$.

To prove part (c), add a state with payoff 0 from which the transition probability to a state in $S_f$ coincides with $\mu$.
Formally, define a degenerate Markov decision process $M' = (S_f \cup \{s^*\},\mu',r',q')$ in which $\mu'$ assigns probability 1 to $s^*$. $r'$
coincides with $r_f$ on $S_f$, while $r'(s^*) := 0$. Finally, $q'$
coincides with $q_f$ on $S_f$, while
at the state $s^*$,
$ q'(s_1 \mid s^*) := \mu(s_1)$.
The value function of $M'$ is $\lambda f(\lambda)$.

To prove part (d), consider the transition function that at every stage, moves to an absorbing state%
\footnote{A state $s \in S$ is \emph{absorbing} if $q(s \mid s,a) = 1$ for every action $a \in A(s)$.}
 with payoff 0 with probability $1-c$,
and with probability $c$ continues as in $M$.
Formally, define a degenerate Markov decision process
$M' = (S_f \cup \{s^*\},\mu, r',q')$ in which $\mu$ coincides with $\mu_f$, $r'$ and $q'$ coincide with $r_f$ and $q_f$ on $S_f$, $r'(s^) := 0$,
and $q'(s^* \mid s^*) := 1$ (that is, $s^*$ is an absorbing state),
and
\[ q'(s^* \mid s) := 1-c, \ \ \ q'(s' \mid s) := cq_f(s' \mid s), \ \ \ \forall s,s' \in S_f. \]
The value function of $M'$ at the initial state $s_{1,f}$ is $f(c\lambda)$.

To prove Part (e), we show that $\tfrac{1}{2}f + \tfrac{1}{2}g$ is in $\calV_D$ and we use part (a) with $a=2$.
The function $\tfrac{1}{2}f+\tfrac{1}{2}g$ is the value function of the degenerate Markov decision process
in which the prior chooses with probability $\tfrac{1}{2}$ one of two degenerate Markov decision processes that implement $f$ and $g$.
Formally, let $M_g = (S_g,\mu_g, r_g,q_g)$ be a degenerate Markov decision process whose value function is $g$.
Let $M = (S_f \cup S_g,\mu',r',q')$
be the degenerate Markov decision process whose state space consists of disjoint copies of $S_f$ and $S_g$,
the functions $r'$ (resp.~$q'$) coincide with $r_f$ and $r_g$ (resp.~$q_f$ and $q_g$) of $S_f$ (resp.~$S_g$,
and the initial distribution is $\mu'=\frac{1}{2}\mu_f+\frac{1}{2}\mu_g$.
The value function of $M$ is $\tfrac{1}{2}f + \tfrac{1}{2}g$.

To prove Part (f), we space out the Markov decision process in a way that stage $k$ of the Markov decision process that implements $f$
becomes stage $1+(k-1)n$, and the payoff in all other stages is 0.
Formally, Let $M' = (S_f \times \{1,2,\cdots,n\},\mu', r',q')$ be a degenerate Markov decision process where $\mu'=\mu_f$ and
\begin{eqnarray*}
q'((s,k+1) \mid (s,k)) := 1 &\ \ \ & k \in \{1,2,\ldots,n-1\}, s \in S,\\
q'(\cdot \mid (s,n)) := q_f(s) &\ \ \ & s \in S,\\
r'((s,1)) := r_f(s) & & s \in S,\\
r'((s,k)) := 0 & & k \in \{2,3,\ldots,n\}, s \in S.
\end{eqnarray*}
The value function of $M'$ with the prior $\mu'$ is $f(\lambda^n)$.
\end{proof}

\bigskip

\begin{lemma}\label{lemma:1a}

\begin{enumerate}
\item[a)] Every polynomial%
\footnote{Throughout the paper, whenever we refer to a polynomial we mean a polynomial with real coefficients.}
$P$ is in $\calV_D$ and if $f \in \calV_D$, then $P\cdot f$ is also in $\calV_D$.
\item[b)] Let $P$ and $Q$ be two polynomials.
If $\tfrac{1}{Q} \in \calV_D$ then $\tfrac{P}{Q} \in \calV_D$.
In particular, if $Q'$ divides $Q$ and $\tfrac{1}{Q} \in \calV_D$ then $\tfrac{1}{Q'} \in \calV_D$.
\item[c)]
If $Q$ is a polynomial whose roots are all unit roots of multiplicity 1, then $\frac{1}{Q}\in \calV_D$.
\end{enumerate}
\end{lemma}
\begin{proof}
Part (a) follows from Lemma \ref{lemma:basic}(a,c,e)
and the observation that any constant function $a$ is in $\calV_D$,
which holds since the constant function $a$ is the value function of the degenerate Markov decision process that starts with
a state whose payoff is $a$ and continues to an absorbing state whose payoff is $0$.

Part (b) follows from Part (a).

We turn to prove Part (c).
The degenerate Markov decision process with a single state in which the payoff is 1 yields payoff $\tfrac{1}{1-\lambda}$,
and therefore $\tfrac{1}{1-\lambda} \in \calV_D$.
By Lemma \ref{lemma:basic}(f) it follows that $\tfrac{1}{1-\lambda^n} \in \calV_D$, for every $n \in \mathbb{N}$.
Let $n$ be large enough such that $Q$ divides $1-\lambda^n$.
The result now follows by Part (b).
\end{proof}

\bigskip

To complete the proof of Theorem \ref{theorem:2} we characterize the polynomials $Q$ that satisfy $\tfrac{1}{Q} \in \calV_D$.
To this end we need the following property of $\calV_D$.

\begin{lemma}
\label{lemma:11}
If $f,g \in \calV_D$ then $f(\lambda)g(\lambda c) \in \calV_D$ for every $c \in (0,1)$.
\end{lemma}

\begin{proof}
The proof of Lemma \ref{lemma:11} is the most intricate part of the proof of Theorem \ref{theorem:1}.
We start with an example, that will help us illustrate the formal definition of the degenerate MDP that implements $f(\lambda)g(\lambda c)$.

Let $M_f$ and $M_g$ be the degenerate Markov decision processes that are depicted in Figure \arabic{figurecounter}
with the initial distributions $\mu_f(s_f) = 1$
and $\mu_g(s_{1,g}) = 1$, in which the payoff at each state appears in a square next to the state.
Denote by $f$ and $g$ the value functions of $M_f$ and $M_g$, respectively.

\bigskip

\centerline{\includegraphics{figure.1} \ \ \ \ \  \ \ \ \ \ \ \ \ \ \includegraphics{figure.2}}

\centerline{Figure \arabic{figurecounter}: An example of two MDP's.}
\addtocounter{figurecounter}{1}
\bigskip

Consider the degenerate Markov decision process $M$  depicted in Figure \arabic{figurecounter},
where $c \in (0,1)$ and the initial state is $s_{1,g}$.

\bigskip

\centerline{\includegraphics{figure.5}}

\centerline{Figure \arabic{figurecounter}: The degenerate MDP $M$.}
\bigskip

\noindent
The MDP $M$ is composed of one copy of $M_g$, and for every state in $S_g$ it contains one copy of $M_f$.
It starts at $s_{1,g}$, the initial state of $M_g$. Then, at every stage, with probability $c$ it continues as in $M_g$,
and with probability $(1-c)$ it moves to a copy of $M_f$.
In case a transition to a copy of $M_f$ occurs,
the new state is chosen according to the transitions $q_f(\cdot\mid s_f)$. This induces a distribution similar to that of the second stage of $M_f$.

The payoff in each of the copies of $M_f$ is the product of the payoff in $M_f$ times the payoff of the state in $M_g$ that has been assigned to that copy.
The payoff in each state $s_g \in S_g$ is $(1-c)r_g(s_g)$ times the expected payoff in the first stage of $M_f$.

Thus, each state in $s_g \in S_g$ serves three purposes (see Figure \arabic{figurecounter}).
First, it is a regular state in the copy of $M_g$.
Second, once a transition from $M_g$ to a copy of $M_f$ occurs (at each stage it occurs with probability $(1-c)$),
it serves as the first stage in $M_f$.
Finally, once a transition from $s_g$ to a copy of $M_f$ occurs, the payoffs in the copy are set to the product of the original payoffs in $M_f$ times $r_g(s_g)$.
\addtocounter{figurecounter}{1}

We now turn to the formal construction of $M$.
Let $f,g \in \calV_D$ and let $M_f=(S_f,\mu_f,r_f,q_f)$ (resp.~$M_g = (S_g, \mu_g, r_g,q_g)$) be the degenerate Markov decision process that implements $f$ (resp.~$g$).
Define the following degenerate Markov decision process $M = (S,\mu,r,q)$:
\begin{itemize}
\item
The set of states is $S = (S_g \times S_f) \cup S_g$.
In words,
the set of states contains a copy of $S_g$,
and for each state $s_g \in S_g$ it contains a copy of $S_f$.

\item
The initial distribution is $\mu_g$:
\[ \mu(s_g) = \mu_g(s_g), \ \ \ \forall s_g \in S_g, \ \ \ \ \ \mu(s_g,s_f) = 0, \ \ \ \forall (s_g,s_f)\in S_g \times S_f. \]

\item
The transition is as follows:
\begin{itemize}
\item
In each copy of $S_f$, the transition is the same as the transition in $M_f$:
\begin{eqnarray}\nonumber
q((s_g,s'_f) \mid (s_g,s_f)) &:=& q_f(s'_f \mid s_f), \ \ \  \forall  s_g \in S_g, s_f,s'_f \in S_f,\\ \nonumber
q((s'_g,s'_f) \mid (s_g,s_f)) &:=& 0, \ \ \ \ \ \ \ \  \forall s_g \neq s'_g \in S_g, s_f,s'_f \in S_f.
\end{eqnarray}
\item
In the copy of $S_g$, with probability $c$ the transition is as in $M_g$,
and with probability $(1-c)$ it is as in $M_f$ starting with the initial distribution $\mu_f$:
\begin{eqnarray*}
q(s'_g \mid s_g) &:=& c q_g(s'_g \mid s_g), \ \ \ \ \ \ \ \ \forall s_g,s'_g \in S_g,\\
q((s_g,s_f) \mid s_g) &:=& (1-c) \sum_{s_f'\in S_f}\mu(s_f')q_f(s_f \mid s_f'), \ \ \  \forall s_g \in S_g, s_f \in S_f,\\
q((s'_g,s_f) \mid s_g) &:=& 0 \ \ \ \ \ \ \ \forall s_g \neq s'_g \in S_g, s_f \in S_f.
\end{eqnarray*}
\end{itemize}
\item
The payoff function is as follows:
\begin{eqnarray*}
r(s_g) &:=& (1-c) r_g(s_g) \sum_{s_f \in S_f} \mu_f(s_f) r_f(s_f), \ \ \ \forall s_g \in S_g,\\
r(s_g,s_f) &:=& r_g(s_g)r_f(s_f), \ \ \ \forall s_g \in S_g, s_f \in S_f.
\end{eqnarray*}
\end{itemize}

\bigskip

\noindent We will now calculate the value function of $M$.
Denote by $$\E_{\mu_f}[r_f] := \sum_{s_f \in S_f} \mu_f(s_f) r_f(s_f)$$ the expected payoff in $M_f$ at the first stage
and by $R$ the expected payoff in $M_f$ from the second stage and on.
Then $f(\lambda) = \E_{\nu_f}[r_f] +\lambda R$.

At every stage, with probability $c$ the process remains in  $S_g$
and with probability $(1-c)$ the process leaves it.
In particular, the probability that at stage $n$ the process is still in $S_g$ is $c^{n-1}$, in which case
(a) the payoff is $(1-c)r_g(s_n)\E_{\mu_f}[r_f]$, and
(b) with probability $(1-c)$ the process moves to a copy of $M_f$, and the total discounted payoff from stage $n+1$ and on is $R$.
It follows that the total discounted payoff is
\begin{eqnarray*}
&&\sum_{n=1}^\infty c^{n-1} \lambda^{n-1} \bigl( (1-c) r_g(s_n) \E_{\mu_f}(r_f) + (1-c)\lambda r_g(s_n)R\bigr)\\
&&= (1-c)\sum_{n=1}^\infty c^{n-1}\lambda^{n-1} r_g(s_k) f(\lambda)\\
&&= (1-c) g(c\lambda) f(\lambda).
\end{eqnarray*}
The result follows by Lemma \ref{lemma:basic}(a).
\end{proof}

\begin{lemma}
\label{lemma:2}
Let $\omega \in \mathbb{C}$ be a complex number that lies outside the unit ball.\footnotemark
\begin{itemize}
\item[a)]
If $\omega \in \mathbb{C}\setminus \mathbb{R}$ then
$\tfrac{1}{(\omega-\lambda)(\overline{\omega}-\lambda)} \in \calV_D$.
\item[b)]
If $\omega \in \mathbb{R}$ then $\tfrac{1}{(\omega-\lambda)} \in \calV_D$.
\end{itemize}
\end{lemma}
\footnotetext{For every complex number $\omega$, the conjugate of $\omega$ is denoted $\overline\omega$.}

\begin{proof}
We start by proving part (a).
For every complex number $\omega \in \mathbb{C}\setminus \mathbb{R}$ that lies outside the unit ball
there are three natural numbers $k < l < m$ and three nonnegative reals $\alpha_1, \alpha_2,\alpha_3$
that sum up to 1 such that $1 = \alpha_1\omega^k + \alpha_2\omega^l + \alpha_3\omega^m$.

Consider the degenerate Markov decision process that is depicted in Figure~\arabic{figurecounter}.
That is, the set of states is $S_f := \{s_1,s_2,\cdots,s_m\}$,
the payoff function is
\[ r(s_m) := 1, \ \ \ \ \ r(s_j):=0, \ \ \ 1 \leq j < m,\]
and the transition function is
\begin{eqnarray*}
&&q(s_{m-k+1} \mid s_m) := \alpha_1, q(s_{m-l+1} \mid s_m) := \alpha_2, , q(s_1 \mid s_m) := \alpha_3, \\
&&q(s_{j+1} \mid s_j) = 1, \ \ \ 1 \leq j < m.
\end{eqnarray*}

\bigskip

\centerline{\includegraphics{figure.4}}

\medskip

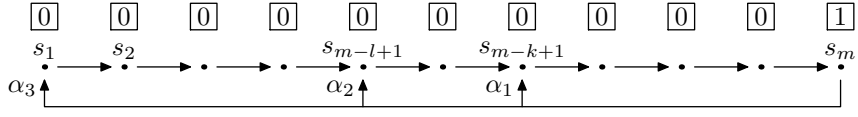
\centerline{Figure \arabic{figurecounter}: The degenerate MDP in the proof of Lemma \ref{lemma:2}.}

\bigskip

The discounted value satisfies
\begin{eqnarray*}
v_\lambda(s_j) &=& \lambda v_\lambda(s_{j+1}), \ \ \ 0 \leq j < m,\\
v_\lambda(s_m) &=& 1 + \lambda \bigl(\alpha_1 v_\lambda(s_{m-k+1}) + \alpha_2 v_\lambda(s_{m-l+1}) + \alpha_3 v_\lambda(s_{1})\bigr).
\end{eqnarray*}
It follows that
\[ v_\lambda(s_m) = \tfrac{1}{1 - \alpha_1 \lambda^{k} -\alpha_2\lambda^{l} - \alpha_3 \lambda^{m}} . \]
Hence, this function  is in $\calV_D$.
Since $\omega$ is one of its roots,
 by Lemma \ref{lemma:1a}(b) we obtain that
$\tfrac{1}{(\omega-\lambda)(\overline{\omega}-\lambda)} \in \calV_D$,
as desired.

\bigskip

We turn to prove part (b).
Let $\omega$ be a real number with $\omega > 1$.
By Lemma \ref{lemma:1a}(b)  $\tfrac{1}{1- \lambda}$ is in $\calV_D$.
By Lemma \ref{lemma:basic}(d), $\tfrac{1}{1- \frac{\lambda}{\omega}}\in\calV_D$,
and by Lemma \ref{lemma:basic}(a), $\tfrac{1}{\omega- \lambda}\in\calV_D$.
 When $\omega < -1$,
by Lemma \ref{lemma:basic}(a,b),  $\tfrac{1}{\omega-\lambda} \in \calV_D$.
\end{proof}
\addtocounter{figurecounter}{1}

\bigskip
\section{The Proof of Theorem \ref{theorem:2}}
Let $Q\not= 0$ be a polynomial with real coefficients whose roots are either outside the unit ball or unit roots with multiplicity 1.
To complete the proof of Theorem \ref{theorem:2} we prove that $\frac{1}{Q}\in \calV_D$.
Denote by $\Omega_1$ the set of all roots of $Q$ that are unit roots,
by $\Omega_2$ the set of all roots of $Q$ that lie outside the unit ball and have a positive imaginary part,
and by $\Omega_3$ the set of all real roots of $Q$ that lie outside the unit ball.
If some roots have multiplicity larger than 1, then they appear several times in $Q_2$ or $Q_3$.

For $i=1,3$ denote $Q_i=\prod_{\omega\in \Omega_i}(\omega-\lambda)$ and set $Q_2= \prod_{\omega\in \Omega_2}(\omega-\lambda)(\overline{\omega}-\lambda)$;
when $\Omega_i=\emptyset$ we set $Q_i=1$.
Then $Q=Q_1\cdot Q_2\cdot Q_3$. If $\Omega_1\not = \emptyset$, then by Lemma \ref{lemma:1a}(c) we have $\frac{1}{Q_1}\in \calV_D$.
Otherwise $Q_1=1$, in which case, $\frac{1}{Q_1}\in \calV_D$ by Lemma \ref{lemma:1a}(a).

Fix $\omega\in \Omega_2$ and let $c \in \mathbb{R}$ be such that $1 < c < |\omega|$.
Since $\frac{\omega }{c}$ lies outside the unit ball, Lemma \ref{lemma:2}(a) implies that
$g_{\omega}(\lambda):=\tfrac{1}{(\frac{\omega}{c} -\lambda)(\overline{\frac{\omega}{c} }-\lambda)}$ is in $\calV_D$.
By Lemma \ref{lemma:11}(a),
$g_{\omega}(\frac{1}{c}\cdot\lambda)\cdot \tfrac{1}{Q_1}= \tfrac{ c^2}{(\omega-\lambda)(\overline{\omega}-\lambda)}\cdot \tfrac{1}{Q_1}\in \calV_D$. By Lemma
\ref{lemma:basic}(a), $\tfrac{1}{(\omega-\lambda)(\overline{\omega}-\lambda)}\cdot \tfrac{1}{Q_1}\in \calV_D$.
Applying successively this argument for the remaining roots in $\Omega_2$, one obtains that $\tfrac{1}{Q_2}\cdot \tfrac{1}{Q_1}\in \calV_D$.

To complete the proof we apply a similar idea to $\omega\in \Omega_3$. Fix $\omega\in \Omega_3$ and let $c \in \mathbb{R}$ be such that $1 < c < |\omega|$. By Lemma \ref{lemma:2}(b),
$\tfrac{1}{\frac{\omega}{c}-\lambda} \in \calV_D$ and again by By Lemma \ref{lemma:11}(a) and Lemma \ref{lemma:basic}(a),
$\tfrac{1}{(\omega-\lambda)}\cdot \tfrac{1}{Q_2}\cdot \tfrac{1}{Q_1}\in \calV_D$.
By iterating this argument for every $\omega\in \Omega_3$ one obtains that $\tfrac{1}{Q_3}\cdot \tfrac{1}{Q_2}\cdot \tfrac{1}{Q_1}\in \calV_D$,
as desired.

\section{Final Remark}

The set $\calV$ contain all value functions of MDP's in which the state at the first stage is chosen according to a probability distribution $\mu$.
One can wonder whether the set of implementable value functions shrinks if one restricts attention to MDP's in which the first stage is given; that is, $\mu$ assigns probability 1 to one state.
The answer is negative: the value function of any MDP in which the initial state is chosen according to a probability distribution (prior)
can be obtained as the value function of an MDP in which the initial state is deterministic.
Indeed, let $M$ be an MDP with a prior. One can construct $M'$ by adding to $M$ an initial state $s'$ in which the payoff is the expected payoff at the first stage of $M$
and the transitions are the expected transitions after the first stage of $M$. We applied a similar idea in the proof of Lemma \ref{lemma:1}.

\end{document}